\documentclass[11pt]{article}

\usepackage{amsmath}
\usepackage{amsthm}
\usepackage{amsfonts}

\newcommand{\mmp}{\mathbb{P}}

\newcommand{\od}{\overset{d}{=}}
\newcommand{\dod}{\overset{d}{\to}}
\newcommand{\tp}{\overset{P}{\to}}

\newcommand{\me}{\mathbb{E}}
\newcommand{\mr}{\mathbb{R}}
\newcommand{\mn}{\mathbb{N}}

\newcommand{\lix}{\underset{x\to\infty}{\lim}}
\newcommand{\lit}{\underset{t\to\infty}{\lim}}

\newtheorem{thm}{Theorem}[section]
\newtheorem{lemma}[thm]{Lemma}

\newtheorem{assertion}[thm]{Proposition}
\theoremstyle{definition}

\theoremstyle{remark}

\begin{document}
\title{On the number of empty boxes in the Bernoulli sieve}
%
\author{Alexander Iksanov\footnote{ Faculty of Cybernetics, National T.
Shevchenko University of Kiev, 01033 Kiev, Ukraine,\newline
e-mail: iksan@unicyb.kiev.ua}}
\maketitle
\begin{abstract}
\noindent The Bernoulli sieve is the infinite "balls-in-boxes"
occupancy scheme with random frequencies $P_k=W_1\cdots
W_{k-1}(1-W_k)$, where $(W_k)_{k\in\mn}$ are independent copies of
a random variable $W$ taking values in $(0,1)$. Assuming that the
number of balls equals $n$, let $L_n$ denote the number of empty
boxes within the occupancy range. The paper proves that, under a
regular variation assumption, $L_n$, properly normalized without
centering, weakly converges to a functional of an inverse stable
subordinator. Proofs rely upon the observation that $(\log P_k)$
is a perturbed random walk. In particular, some results for
general perturbed random walks are derived. The other result of
the paper states that whenever $L_n$ weakly converges (without
normalization) the limiting law is mixed Poisson.
\end{abstract}


\section{Introduction and main results}

The {\it Bernoulli sieve} is the infinite occupancy scheme with
random frequencies
\begin{equation}\label{17}
P_k:=W_1W_2\cdots W_{k-1}(1-W_k), \ \ k\in\mn,
\end{equation}
where $(W_k)_{k\in\mn}$ are independent copies of a random
variable $W$ taking values in $(0,1)$. This phrase should be
interpreted in the sense that the balls are allocated over an
infinite array of boxes $1,2,\ldots$ independently conditionally
given $(P_k)$ with probability $P_j$ of hitting box $j$.

Assuming that the number of balls equals $n$ denote by $K_n$ the
number of occupied boxes, $M_n$ the index of the last occupied
box, and $L_n:=M_n-K_n$ the number of empty boxes within the
occupancy range. The main purpose of this note is to prove the
following theorem which is the first result about the weak
convergence of, properly normalized, $L_n$ in the case of
divergence $L_n\tp \infty$, $n\to\infty$. Unlike all the previous
papers about the Bernoulli sieve (see \cite{GIM2} for a survey)
lattice laws of $|\log W|$ are allowed.
\begin{thm}\label{main}
Suppose
$$\mmp\{|\log W|>x\} \ \sim \ x^{-\alpha}\ell_1(x) \ \ \text{and} \ \ \mmp\{|\log (1-W)|>x\} \ \sim \ x^{-\beta}\ell_2(x),
\ \ x\to\infty,$$ for some $0\leq \beta\leq \alpha<1$
($\alpha+\beta>0$) and some $\ell_1$ and $\ell_2$ slowly varying
at $\infty$. If $\beta=\alpha>0$ it is additionally assumed that
$$\lix {\mmp\{|\log W|>x\}\over \mmp\{|\log (1-W)|>x\}}=0.$$ Then
$${\mmp\{W\leq 1/n\}\over \mmp\{1-W\leq 1/n\}}L_n \ \dod \ \int_0^1 (1-s)^{-\beta}{\rm d}X_\alpha^\leftarrow(s)
=:Z_{\alpha, \beta}, \ \ n\to\infty,$$ where
$(X_\alpha^\leftarrow(t))_{t\geq 0}$ is an inverse $\alpha$-stable
subordinator defined by
$$X_\alpha^\leftarrow(t):=\inf\{s\geq 0: X_\alpha(s)>t\},$$ where $(X_\alpha(t))_{t\geq 0}$ is an
$\alpha$-stable subordinator with $-\log \me e^{-sX_\alpha(1)}=
\Gamma(1-\alpha)s^\alpha$, $s\geq 0$.
\end{thm}


Turning to the case of convergence
\begin{equation}\label{con}
L_n\dod L, \ \ n\to\infty,
\end{equation}
where $L$ is a random variable with proper and nondegenerate
probability law, we start by recalling some previously known facts
which can also be found in \cite{GIM2}.
\begin{itemize}
\item {\rm Fact 1}. Relation \eqref{con} holds true whenever $\me
|\log W|<\infty$, $\me |\log (1-W)|<\infty$, and the law of $|\log
W|$ is nonlattice. In this case $L$ has the same law as the number
of empty boxes in a limiting scheme with infinitely many balls
(see \cite{GIM2} for more details).

\item {\rm Fact 2}.\label{fa} If $W\od 1-W$ then, for every
$n\in\mn$, $L_n$ has the geometric law starting at zero with
success probability $1/2$.

\item {\rm Fact 3}. If $W$ has beta $(\theta, 1)$ law ($\theta>0$)
then $L$ has the mixed Poisson distribution with the parameter
distributed like $\theta|\log(1-W)|$.
\end{itemize}
As a generalization of the last two facts we prove the following.
\begin{assertion}\label{mix}
Whenever \eqref{con} holds true, $L$ has a mixed Poisson law.
\end{assertion}
The number $L_n$ of empty boxes in the Bernoulli sieve can be
thought of as the number of zero decrements before the absorption
of certain nonincreasing Markov chain starting at $n$. Proposition
\ref{mix} will be established in Section \ref{mar} as a corollary
to a more general result formulated in terms of arbitrary
nonincreasing Markov chains.

With the exception of Section 5 the rest of the paper is organized
as follows. In Section 2 we identify the law of the random
variable $Z_{\alpha,\beta}$ defined in Theorem \ref{main}. Section
3 investigates the weak convergence of a functional defined in
terms of general perturbed random walks. Specializing this result
to the particular perturbed random walk $(|\log P_k|)$ Theorem
\ref{main} is then proved in Section 4.

\section{Identification of the law of $Z_{\alpha, \beta}$}

In this section we will identify the law of the random variable
$$Z_{\alpha, \beta}=\int_0^1 (1-s)^{-\beta}{\rm
d}X_\alpha^\leftarrow(s)$$ arising in Theorem \ref{main}. Since
the process $(X_\alpha^\leftarrow(t))_{t\geq 0}$ has nondecreasing
paths, the integral is interpreted as a pathwise
Lebesgue-Stieltjes integral.

Let $T$ be a random variable with the standard exponential law
which is independent of $(Y_\alpha(t))_{t\geq 0}$ a drift-free
subordinator with no killing and the L\'{e}vy measure
$$\nu_\alpha({\rm d}t)={e^{-t/\alpha}\over (1-e^{-t/\alpha})^{\alpha+1}}1_{(0,\infty)}(t){\rm d}t.$$
One can check that $$\Phi_\alpha(x):=-\log \me
e^{-xY_\alpha(1)}={\Gamma(1-\alpha)\Gamma(\alpha x+1)\over
\Gamma(\alpha(x-1)+1)}-1, \ \ x\geq 0.$$ It is known (see, for
instance, \cite{IksMoe2}) that
$$X^\leftarrow_\alpha(1)\od \int_0^T e^{-Y_\alpha(t)}{\rm d}t.$$
and that $X^\leftarrow_\alpha(1)$ has a Mittag-Leffler law which
is uniquely determined by its moments
\begin{equation}\label{mome1}
\me (X^\leftarrow_\alpha(1))^n ={n!\over (\Phi_\alpha(1)+1)\ldots
(\Phi_\alpha(n)+1)}={n!\over \Gamma(1+n\alpha)\Gamma^n(1-\alpha)},
\ \ n\in\mn.
\end{equation}
According to \cite[p.~3245]{Mag}, $$\me Z_{\alpha,
\beta}^n={n!\alpha^n \over
\Gamma^n(1-\alpha)\Gamma^n(1+\alpha)}\int_0^t\int_0^{t_1}\ldots\int_0^{t_{n-1}}\prod_{i=1}^n
(1-t_i)^{-\beta}(t_i-t_{i+1})^{\alpha-1}{\rm d}t_n\ldots{\rm
d}t_1,$$ where $t_{n+1}=0$. Note that our setting requires an
additional factor. Changing the order of integration followed by
some calculations lead to
\begin{eqnarray}\label{mome}
\me Z_{\alpha, \beta}^n&=&{n!\over (\Phi_\alpha(c)+1)\ldots
(\Phi_\alpha(cn)+1)}\nonumber\\&=& {n!\over
\prod_{k=1}^n(1-\alpha+k(\alpha-\beta)){\rm B}(1-\alpha,
1+k(\alpha-\beta))}, \ \ n\in\mn.
\end{eqnarray}
where $c:=(\alpha-\beta)/\alpha$, which, by \cite[Theorem
2(i)]{BerYor}, entails the distributional equality
\begin{equation}\label{127}
Z_{\alpha, \beta}\od \int_0^T e^{-cY_\alpha(t)}{\rm d}t.
\end{equation}
From the inequality $\int_0^T e^{-cY_\alpha(t)}{\rm d}t\leq T$ and
the fact that $\me e^{aT}<\infty$, for $a\in (0,1)$ (or just using
the last cited result) we conclude that the law of $Z_{\alpha,
\beta}$ has some finite exponential moments and thereby is
uniquely determined by its moments.

Since $$X_\alpha(t)=\inf\{s\geq 0: L_\alpha(s)>t\},$$ where
$(L_\alpha(t))_{t\geq 0}$ is a local time at level $0$ for the
$2(1-\alpha)$-dimensional Bessel process (see \cite[p.~555]{Yor}),
we observe that
$$Z_{\alpha, \beta}=\int_0^1 (1-s)^{-\beta}{\rm d}L_\alpha(s).$$ Therefore formula \eqref{mome} can alternatively
be obtained from \cite[formula (4.3)]{Yor}.

Finally, two particular cases are worth mentioning
$$Z_{\alpha, 0}\od X^\leftarrow_\alpha(1) \ \ \text{and} \ \ Z_{\alpha,
\alpha}\od T.$$

\section{Results for perturbed random walks}\label{prw}

Let $(\xi_k, \eta_k)_{k\in\mn}$ be independent copies of a random
vector $(\xi, \eta)$ with $\xi\geq 0$, $\eta\geq 0$ and
$\mmp\{\xi=0\}<1$. We make no assumptions about dependence
structure of $(\xi, \eta)$. Denote by $F$ and $G$ the distribution
functions of $\xi$ and $\eta$, respectively. Let
$(S_n)_{n\in\mn_0}$ be a zero-delayed random walk with a step
distributed like $\xi$. A sequence $(T_n)_{n\in\mn}$ defined by
$$T_n:=S_{n-1}+\eta_n, \ \ n\in\mn,$$ will be called a {\it perturbed random
walk}. 

Set
$$\rho(t):=\#\{k\in\mn_0: S_k\leq t\}\ = \ \inf\{k\in\mn: S_k>t\},
\ \ t\geq 0,$$ and $U(t):=\me \rho(t)$. Then $U(t)$ is the renewal
function of $(S_k)$. We want to investigate the weak convergence
of
$$T(t):=\sum_{k\geq 1}\bigg(\exp(-te^{-T_k})-\exp(-te^{-S_{k-1}})\bigg), \ \ t\geq 0.$$ As it will become clear in Section \ref{s3},
this problem is relevant to proving Theorem \ref{main}.

We need two technical results the first of which is an essential
improvement over \cite[Theorem 1]{Ath}.
\begin{lemma}\label{tech}
Suppose
\begin{equation}\label{sss1}
1-F(x)\ \sim \ x^{-\alpha}\ell_1(x), \ \ x\to\infty,
\end{equation}
and let $Q$ be a nonincreasing function such that $Q(0)<\infty$
and
\begin{equation}\label{sss2}
Q(x)\ \sim \ x^{-\beta}\ell_2(x), \ \ x\to\infty,
\end{equation}
for some $0\leq \beta\leq \alpha<1$ and some $\ell_1$ and $\ell_2$
slowly varying at $\infty$. Then
\begin{equation}\label{im}
{1-F(t)\over Q(t)}\int_0^t  Q(t-x)\rho({\rm d}x)\ \dod \
Z_{\alpha, \beta}, \ \ t\to\infty,
\end{equation}
along with convergence of expectations
\begin{eqnarray}\label{789}
{1-F(t)\over Q(t)}\me \int_0^t Q(t-x)\rho({\rm d}x)&=&{1-F(t)\over
Q(t)}\int_0^t Q(t-x){\rm d}U(x)\\ &\to& \me Z_{\alpha,
\beta}={1\over (1-\beta)B(1-\alpha, 1+\alpha-\beta)}.\nonumber
\end{eqnarray}
\end{lemma}
\begin{proof}
It is well-known \cite[Theorem 1b]{Bing73} that condition
\eqref{sss1} entails
$$(1-F(t))\rho(t\cdot) \ \Rightarrow \ X^\leftarrow_\alpha(\cdot), \ \ t\to\infty,$$
under $M_1$ topology in $D[0,1]$. The one-dimensional convergence
holds along with convergence of all moments. In particular,
\begin{equation}\label{co}
\lit (1-F(t))U(t)=\me
X^\leftarrow_\alpha(1)\overset{\eqref{mome1}}{=}{1\over
\Gamma(1-\alpha)\Gamma(1+\alpha)}=:c_\alpha.
\end{equation}
The Skorohod's representation theorem ensures the existence of
versions $(\widehat{\rho}(t))\od (\rho(t))$ and
$(\widehat{X}^\leftarrow_\alpha(t))\od (X^\leftarrow_\alpha(t))$
such that
$$(1-F(t))\widehat{\rho}(t\cdot) \ \overset{{\rm a.s.}}{\to} \ \widehat{X}^\leftarrow_\alpha(\cdot), \ \
t\to\infty.$$ Furthermore, we can assume that
$(\widehat{\rho}(t))$ is a.s. nondecreasing and that
$(\widehat{X}^\leftarrow_\alpha(t))$ is a.s. continuous as it is
the case for the original processes. In particular, with
probability one, $$\lit
(1-F(t))\big(\widehat{\rho}(t)-\widehat{\rho}(t(1-s)\big)=\widehat{X}^\leftarrow_\alpha(1)-\widehat{X}^\leftarrow_\alpha(1-s)$$
uniformly on $[0,1]$, as it is the convergence of monotone
functions to a continuous limit. Further, by virtue of
\eqref{sss2}, the measure $\mu_t$ defined by
$\mu_t((s,1]):={Q(ts)\over Q(t)}$, $s\in [0,1)$ converges weakly,
as $t\to\infty$, to a measure with density $s\to \beta
s^{-\beta-1}$, $s\in (0,1)$. Hence, setting $I(t):=\int_0^t Q(t-x)
\rho({\rm d}x)$, we have, as $t\to\infty$,
\begin{eqnarray}\label{1221}
{1-F(t)\over Q(t)}I(t) &=& \int_0^1 {\rho(t)-\rho(t(1-s))\over
(1-F(t))^{-1}}\mu_t({\rm d}s)+(1-F(t))\rho(t)\nonumber\\&\od&
\int_0^1 {\widehat{\rho}(t)-\widehat{\rho}(t(1-s))\over
(1-F(t))^{-1}}\mu_t({\rm
d}s)+(1-F(t))\widehat{\rho}(t)\nonumber\\&\overset{{\rm
a.s.}}{\to}& \int_0^1
(\widehat{X}^\leftarrow_\alpha(1)-\widehat{X}^\leftarrow_\alpha(1-s))\beta
s^{-\beta-1}{\rm d}s+\widehat{X}^\leftarrow_\alpha(1)\\&=&
\int_0^1 (1-s)^{-\beta}{\rm
d}\widehat{X}^\leftarrow_\alpha(s)\nonumber\\&\od&
Z_{\alpha,\beta}\nonumber,
\end{eqnarray}
which proves \eqref{im}.

In view of \eqref{co}, $\lit {U(t)-U(t(1-s))\over
U(t)}=1-(1-s)^\alpha$. Furthermore, the convergence is uniform on
$[0,1]$. Now setting $J(t):=\int_0^t Q(t-x) {\rm d}U(x)$ and
arguing in the same way as above we conclude that, as
$t\to\infty$,
\begin{eqnarray*}
{1-F(t)\over Q(t)}J(t) &=& (1-F(t))U(t)\bigg(\int_0^1
{U(t)-U(t(1-s))\over U(t)}\mu_t({\rm d}s)+1\bigg)\\&\to&
c_\alpha\bigg(\int_0^1 (1-(1-s)^\alpha)\beta s^{-\beta-1}{\rm d}s
+1\bigg)\\&=& c_\alpha\int_0^1(1-s)^{-\beta}\alpha
s^{\alpha-1}{\rm d}s\overset{\eqref{mome}}{=}\me Z_{\alpha,
\beta},
\end{eqnarray*}
which proves \eqref{789}.
\end{proof}
\begin{lemma}\label{aux}
Let $\Psi(s)$ be the Laplace transform of a random variable
$\theta$ taking values in $[0,1]$ with $\mmp\{\theta>0\}>0$. The
following functions
$$f_1(t):=\Psi(e^t)-\Psi(2e^t), \ \ f_2(t):=(1-\Psi(2e^t))1_{(-\infty,\, 0)}(t),$$
$$f_3(t):=(1-\Psi^2(e^t))1_{(-\infty,\, 0)}(t), \ \ f_4(t):=(\Psi(e^t)-\exp(-e^t))1_{(-\infty,\, 0)}(t),$$
$$f_5(t):=\exp(-e^t)(\Psi(e^t)-\exp(-e^t)), \ \ f_6(t):=\exp(-e^t)(1-\Psi(e^t))$$ and
$$f_7(t):=\exp(-e^t)\,\me \sum_{k\geq
1}\bigg(\Psi\big(e^{t-S_k}\big)-\exp\big(-e^{t-S_k}\big)\bigg)$$
are directly Riemann integrable on $\mr$. In particular,
\begin{equation}\label{integr}
\lit \int_\mr f_k(t-x){\rm d}U(x)=0, \ \ k=1,2,3,4,5,6,7,
\end{equation}
whenever $\me \xi=\infty$.
\end{lemma}
\begin{proof}
$k=1,2,3,4,5,6$: The functions $f_k(t)$ are nonnegative and
integrable on $\mr$. Indeed, while $\int_\mr f_1(t){\rm d}t=\log
2$, for $k=2,3,4,5,6$ the integrability is secured by the
finiteness of $\Psi^\prime(0)$. Furthermore, the functions
$t\mapsto e^{-t}f_k(t)$ are nonincreasing. It is known that these
properties together ensure that the $f_k(t)$ are directly Riemann
integrable (see, for instance, the proof of \cite[Corollary
2.17]{Durr}). Finally, an application of the key renewal
theorem\footnote{When the law of $\xi$ is $d$-lattice, the
standard form of the key renewal theorem proves \eqref{integr}
with the limit taken over $t\in d\mn$. Noting that in the case
$\me \xi=\infty$ $\lit (U(t+y)-U(t))=0$, for any $y\in\mr$, and
using Feller's classical approximation argument (see p.~361-362 in
\cite{Feller}) lead to \eqref{integr}.} yields
\eqref{integr}.\newline $k=7$: The integrability of $f_4(t)$
ensures that $f_7(t)$ is finite. The integrability of $f_7(t)$ on
$\mr$ is equivalent to the integrability of $t\mapsto f_7(\log
t)/t$ on $[0,\infty)$. Now the integrability of the latter
function at the neighborhood of $0$ follows from the inequality
$${f_7(\log t)\over t} \leq {\sum_{k\geq
1}\bigg(\Psi(te^{-S_k})-\exp(-te^{-S_k})\bigg)\over t}\leq \me
(1-\theta)\me \sum_{k\geq 1}e^{-S_k}<\infty,$$ which holds for
$t>0$. The relation $f_7(\log t)/t=o(e^{-t})$, $t\to\infty$, which
follows from \eqref{integr} for $f_4$, ensures the integrability
at the neighborhood of $+\infty$. To prove that $t\mapsto
e^{-t}f_7(t)$ is nonincreasing it suffices to check that the
function $t\mapsto e^{-t}{\Psi(at)-\exp(-at)\over t}$, for any
fixed $a>0$, is nonincreasing. This is easy as the function
$t\mapsto e^{-(a\theta +1)t}{1-e^{-a(1-\theta)t}\over t}$ is
completely monotone, hence nonincreasing, and passing to the
expectations preserve the monotonicity. An appeal to the key
renewal theorem proves \eqref{integr} for $f_7$.
\end{proof}

\begin{thm}\label{ma}
Suppose
\begin{equation}\label{sss}
1-F(x)\ \sim \ x^{-\alpha}\ell_1(x) \ \ \text{and} \ \ 1-G(x) \
\sim \ x^{-\beta}\ell_2(x), \ \ x\to\infty,
\end{equation}
for some $0\leq \beta\leq \alpha<1$ and some $\ell_1$ and $\ell_2$
slowly varying at $\infty$. If $\beta=\alpha$ it is additionally
assumed that $\lix {1-F(x)\over 1-G(x)}=0$. Then
$${1-F(\log t)\over 1-G(\log t)}T(t) \ \dod \ Z_{\alpha,\beta}, \ \ t\to\infty,$$ where the random variable
$Z_{\alpha, \beta}$ was defined in Theorem \ref{main}.
\end{thm}
\begin{proof}

Set $\varphi(t):=\me \exp(-te^{-\eta})$,
$\psi_1(t):=\varphi(2e^t)$, $\psi_2(t):=\varphi^2(e^t)$,
$\psi(t):=\psi_1(t)-\psi_2(t)$ and, for $k\in\mn$,
$\mathcal{F}_k:=\sigma((\xi_j, \eta_j):j\leq k)$, $\mathcal{F}_0$
being the trivial $\sigma$-algebra. We will show that a major part
of the variability of $T(t)$ is absorbed by a {\it renewal
shot-noise} process $(V(t))_{t\geq 0}$, where
\begin{eqnarray*}
V(t)&:=&\sum_{k\geq 0}\me \big(\exp(-t
e^{-T_{k+1}})-\exp(-te^{-S_k})|\mathcal{F}_k\big)\\&=& \sum_{k\geq
0} \big(\varphi(t e^{-S_k})-\exp(-te^{-S_k})\big).
\end{eqnarray*}
More precisely, we will prove that, as $t\to\infty$,
\begin{equation}\label{131}
{1-F(\log t)\over 1-G(\log t)}\bigg(T(t)-V(t)\bigg) \tp \ 0 \ \
\text{and} \ \ {1-F(\log t)\over 1-G(\log t)}V(t) \ \dod \
Z_{\alpha, \beta}.
\end{equation}
It can be checked that
$$q(t):=\me \bigg(T(t)-V(t)\bigg)^2=\sum_{k\geq 0}\bigg(\varphi(2t e^{-S_k})-\varphi^2(te^{-S_k})\bigg).$$
Hence, $q(e^t)=\int_0^\infty \psi(t-z){\rm d}U(z)$.

The second condition in \eqref{sss} is equivalent to
$$\mmp\{e^{-\eta}\leq y\} \ \sim \ (\log
(1/y))^{-\beta}\ell_2(\log (1/y)), \ \ y\to +0,$$ hence to
$$\psi_1(t) \ \sim \ t^{-\beta}\ell_2(t) \ \ \text{and} \ \ \psi_2(t) \ \sim \ t^{-2\beta}\ell_2^2(t), \ \ t\to\infty,$$
by \cite[Theorem 1.7.1']{BGT}. With this at hand, applying Lemma
\ref{tech} separately to the integrals $\int_0^t \psi_1(t-x){\rm
d}U(x)$ and $\int_0^t \psi_2(t-x){\rm d}U(x)$ yields
\begin{equation}\label{xx}
\int_0^t \psi(t-x){\rm d}U(x) \ \sim \ {\rm const}{1-G(t)\over
1-F(t)}, \ \ t\to\infty.
\end{equation}
Under the present assumptions $\me \xi=\infty$. Therefore, using
Lemma \ref{aux} for $f_2$ and $f_3$ gives $\lit \int_t^\infty
\psi(t-x){\rm d}U(x)=0$. Now combining this with \eqref{xx} leads
to $$q(t) \ \sim \ {\rm const}{1-G(\log t)\over 1-F(\log t)}, \ \
t\to\infty,$$ and the first part of \eqref{131} follows by
Chebyshev's inequality.

Applying Lemma \ref{aux} for $f_4$ allows us to conclude that
$$\lit \me \sum_{k\geq
0}\bigg(\varphi(e^{t-S_k})-\exp(-e^{t-S_k})\bigg)1_{\{S_k>t\}}=0.$$
Thus the second part of \eqref{131} is equivalent to
\begin{eqnarray}\label{156}
&& {1-F(t)\over 1-G(t)}\sum_{k\geq
0}\bigg(\varphi(e^{t-S_k})-\exp(-e^{t-S_k})\bigg)1_{\{S_k\leq
t\}}\nonumber\\&=& {1-F(t)\over 1-G(t)}\int_0^t
\bigg(\varphi(e^{t-x})-\exp(-e^{t-x})\bigg)\rho({\rm d}x)\nonumber
\\&\dod& Z_{\alpha, \beta}, \ \ t\to\infty.
\end{eqnarray}
Since the function $z(t):=\varphi(e^t)-\me e^{-\eta}\exp(-e^t)$ is
nonincreasing\label{12121} and $z(x) \ \sim \ 1-G(x) \ \sim \
x^{-\beta}\ell_2(x)$, $x\to\infty$, an application of Lemma
\ref{tech} gives $${1-F(t)\over 1-G(t)}\int_0^t z(t-x) \rho({\rm
d}x)\ \dod \ Z_{\alpha, \beta}, \ \ t\to\infty.$$ Similarly,
\begin{equation}\label{155}
{1-F(t)\over 1-G(t)}\int_0^t \varphi(e^{t-x}) \rho({\rm d}x)\ \dod
\ Z_{\alpha, \beta}, \ \ t\to\infty.
\end{equation}
Consequently,
$${1-F(t)\over 1-G(t)}\int_0^t \exp(-e^{t-x})\rho({\rm d}x)\ \tp \
0, \ \ t\to\infty,$$ which together with \eqref{155} proves
\eqref{156} and hence the theorem.
\end{proof}

Another interesting functional of the perturbed random walk
$(T_k)$ is
$$R(t):=\sum_{k=0}^\infty 1_{\{S_k\leq t<S_k+\eta_{k+1}\}}, \ \ t\geq 0.$$
Note that $(R(t))_{t\geq 0}$ is a shot-noise process which has
received some attention in the recent literature. Assuming that
$\xi$ and $\eta$ are independent the process was used to model the
number of busy servers in the ${\rm GI}/{\rm G}/\infty$ queue
\cite{Kaplan} and the number of active sessions in a computer
network \cite{Konst, Res}.

Arguing in a similar but simpler way as in the proof of Theorem
\ref{ma} we can prove the following.
\begin{assertion}
Under the assumptions of Theorem \ref{ma},
$${1-F(t)\over 1-G(t)}R(t) \ \dod \ Z_{\alpha,\beta}, \ \ t\to\infty.$$
\end{assertion}
In a simpler case that $\xi$ and $\eta$ are independent another
proof of this result was given in \cite{Res}.

\section{Proof of Theorem \ref{main}}\label{s3}

Set
$$S^\ast_0:=0 \ \ \text{and} \ \ S^\ast_k:=|\log
W_1|+\ldots+|\log W_k|, \ \ k\in\mn,$$
$$T^\ast_k:=S^\ast_{k-1}+|\log (1-W_k)|=\log P_k, \ \ k\in\mn,$$ $$F^\ast(x):=\mmp\{|\log W|\leq x\} \ \ \text{and} \ \
G^\ast(x):=\mmp\{|\log(1-W)|\leq x\},$$ and $\varphi^\ast(t):=\me
e^{-t|\log(1-W)|}$. Since the sequence $(T^\ast_k)_{k\in\mn}$ is a
perturbed random walk, the results developed in the previous
section can be applied now.


It is clear that the numbers of balls hitting different boxes in
the Bernoulli sieve are dependent. Throwing balls at the epochs of
a unit rate Poisson process $(\pi_t)_{t\geq 0}$ leads to a
familiar simplification. Indeed, denoting by $\pi_t^{(k)}$ the
number of points falling in box $k$ we conclude that,
conditionally on $(P_j)$, the $(\pi_t^{(k)})_{t\geq 0}$ are
independent Poisson processes with rates $P_k$'s, and
$\pi_t=\sum_{k\geq 1}\pi^{(k)}_t$. The replacement of $n$ by
$\pi_t$ is called {\it poissonization} and in the present setting
it reduces investigating $L_n$ to studying $L(t):=L_{\pi_t}$,
where the indices are independent of $(L_j)$.

In the Bernoulli sieve the variability of the allocation of balls
is affected by both randomness in sampling and randomness of
frequencies $(P_k)$. Our first key result states that with respect
to the number of empty boxes the sampling variability is
negligible in a strong sense whenever the expectation of $|\log
W|$ is infinite.
\begin{lemma}\label{imp}
Whenever $\me |\log W|=\infty$, $$L(t)-\me (L(t)|(P_k)) \ \tp \ 0,
\ \ t\to\infty.$$
\end{lemma}
\begin{proof}
Using Chebyshev's inequality followed by passing to expectations
we conclude that it suffices to prove that
\begin{equation}\label{223}\underset{t\to\infty}{\lim}\me\,{\rm Var}(L(t)|(P_k))=0.
\end{equation}
Since $L(t)=\sum_{k\geq 1}1_{\{\pi_t^{(k)}=0,\, \sum_{i\geq
k+1}\pi_t^{(i)}\geq 1\}}$ we have
\begin{eqnarray*}
{\rm Var}\,(L(t)|(P_k))&=&\sum_{k\geq
1}\bigg(e^{-tP_k}-e^{-2tP_k}\bigg)\\&+& \sum_{k\geq 1}e^{-t(1-P_1-\ldots-P_{k-1})}\bigg(2e^{-tP_k}-e^{-t(1-P_1-\ldots-P_{k-1})}-1\bigg) \nonumber\\
&+& 2\sum_{1\leq i<j}
e^{-t(1-P_1-\ldots-P_{i-1})}\bigg(e^{-tP_j}-e^{-t(1-P_1-\ldots-P_{j-1})}\bigg)\nonumber\\
&=&\sum_{k\geq
1}\bigg(\exp(-te^{-T_k^\ast})-\exp(-2te^{-T^\ast_k})\bigg)\nonumber\\&+&
\sum_{k\geq
1}\exp\big(-te^{-S_{k-1}^\ast}\big)\bigg(2\exp(-te^{-T_k^\ast})-\exp(-te^{-S_{k-1}^\ast})-1\bigg)
\nonumber\\&+& 2\sum_{1\leq i<j}
\exp\big(-te^{-S^\ast_{i-1}}\big)\bigg(\exp(-te^{-T_j^\ast})-
\exp(-te^{-S_{j-1}^\ast})\bigg)\\&=:& y_1(t)+y_2(t)+2y_3(t).
\end{eqnarray*}

Setting $$f_1^\ast(t):=\varphi^\ast(e^t)-\varphi^\ast(2e^t), \ \
f_5^\ast(t):=\exp(-e^t)\big(\varphi^\ast(e^t)-\exp(-e^t)\big)$$
and
$$f_6^\ast(t):=\exp(-e^t)\big(1-\varphi^\ast(e^t)\big), \ \ f_7^\ast(t):=\exp(-e^t)\,\me \sum_{k\geq 1}\bigg(\varphi^\ast
\big(e^{t-S^\ast_k}\big)-\exp\big(-e^{t-S^\ast_k}\big)\bigg),$$
one can check that
$$\me y_1(e^t)=\int_\mr f_1^\ast(t-x){\rm d}U^\ast(x), \ \ \me y_2(e^t)=\int_\mr \big(f_5^\ast(t-x)-f_6^\ast(t-x)\big){\rm d}U^\ast(x)$$ and
$$\me y_3(e^t)=\int_\mr f_7^\ast(t-x){\rm d}U^\ast(x).$$ By using
\eqref{integr} for $f_1$, $f_5$, $f_6$ and $f_7$ we conclude that
either of these expectations goes to zero, as $t\to\infty$,
thereby proving \eqref{223} and hence the lemma.
\end{proof}
Observe now that
\begin{eqnarray}\label{sas}
\me (L(t)|(P_k))&=&\sum_{k\geq
1}\bigg(e^{-tP_k}-e^{-t(1-P_1-\ldots-P_{k-1})}\bigg)\\&=&\sum_{k\geq
1}\bigg(\exp(-te^{-T_k^\ast})-\exp(-te^{-S^\ast_{k-1}})\bigg),
\end{eqnarray}
which is a particular instance of the functional $T(t)$ (see
Section \ref{prw}). Assuming that the assumptions of Theorem
\ref{main} hold we then conclude, by Theorem \ref{ma}, that, with
$a(t):={1-G^\ast(\log t)\over 1-F^\ast(\log t)}$,
\begin{equation*}\label{135}
\me(L(t)|(P_k))/a(t) \ \dod \ Z_{\alpha, \beta}, \ \ t\to\infty.
\end{equation*}
An appeal to Lemma \ref{imp} proves
$$L(t)/a(t) \ \dod \ Z_{\alpha, \beta}, \ \ t\to\infty.$$

It remains to pass from the poissonized occupancy model to the
fixed-$n$ model.

For any fixed $\epsilon\in (0,1)$ and $x>0$ we have
\begin{eqnarray*}
\mmp\{L(t)/a_\varepsilon(t)>x\} & \leq &
\mmp\{L(t)/a_\varepsilon(t)>x, \lfloor(1-\epsilon)t\rfloor\leq
\pi_t\leq \lfloor(1+\epsilon)t\rfloor\}+
\mmp\{|\pi_t-t|>\epsilon t\} \\
& \leq & \mmp\{\underset{\lfloor(1-\epsilon)t\rfloor\leq i\leq
\lfloor (1+\epsilon)t\rfloor}{\max}
L_i/a_\varepsilon(t) >x\}+ \mmp\{|\pi_t-t|>\epsilon t\}\\
& = & \mmp\{L_{\lfloor
(1-\epsilon)t\rfloor}/a_\varepsilon(t)>x\}+\mmp\{L_{\lfloor
(1-\epsilon)t\rfloor}\leq a_\varepsilon(t)x, \
\underset{\lfloor(1-\epsilon)t\rfloor +1\leq i\leq
\lfloor(1+\epsilon)t\rfloor}{\max} L_i>a_\varepsilon(t)x\}\\
& + & \mmp\{|\pi_t-t|>\epsilon t\}:=I_1(t)+I_2(t)+I_3(t),
\end{eqnarray*}
where $a_\varepsilon(t):=a(\lfloor(1-\epsilon)t\rfloor)$.
Similarly,
\begin{eqnarray}\label{1}
\mmp\{L(t)/\widehat{a}_\varepsilon(t) \leq x\} & \leq &
\mmp\{L_{\lfloor
(1+\epsilon)t\rfloor}/\widehat{a}_\varepsilon(t)\leq
x\}+\mmp\{L_{\lfloor(1+\epsilon)t\rfloor}>\widehat{a}_\varepsilon(t)x,
\ \underset{\lfloor (1-\epsilon)t\rfloor \leq i\leq
\lfloor(1+\epsilon)t\rfloor-1 }{\min} L_i \leq
\widehat{a}_\varepsilon(t)
x\}\notag \\
& + & \mmp\{|\pi_t-t|>\epsilon t\}:=J_1(t)+J_2(t)+I_3(t),
\end{eqnarray}
where
$\widehat{a}_\varepsilon(t):=a(\lfloor(1+\epsilon)t\rfloor)$.

It is known \cite{GINR} that frequencies \eqref{17} can be
considered as the sizes of the component intervals obtained by
splitting $[0,1]$ at points of the multiplicative renewal process
$(Q_k)_{k\in\mn_0}$, where
$$Q_0:=1, \ \ Q_j:=\prod_{i=1}^j W_i, \ \ j\in\mn.$$ Accordingly, the boxes can be identified with
open intervals $(Q_k, Q_{k-1})$, and balls with points of an
independent sample $U_1,\ldots,U_n$ from the uniform distribution
on $[0,1]$ which is independent of $(Q_k)$. In this representation
balls $i$ and $j$ occupy the same box iff points $U_i$ and $U_j$
belong to the same component interval.

If uniform points $U_{[(1-\epsilon)t]+1}$, \ldots, $U_{\lfloor
(1+\epsilon)t\rfloor }$ fall to the right from the point
$\underset{1\leq i\leq \lfloor (1-\epsilon)t\rfloor}{\min}\,U_i$
then
$$\underset{\lfloor(1-\epsilon)t\rfloor +1\leq i\leq
\lfloor(1+\epsilon)t\rfloor}{\max} L_i \leq L_{\lfloor
(1-\epsilon)t\rfloor} \ \ \text{and} \ \ L_{\lfloor
(1+\epsilon)t\rfloor}\leq
\underset{\lfloor(1-\epsilon)t\rfloor\leq i\leq
\lfloor(1+\epsilon)t\rfloor-1}{\min} L_i,$$ which means that
neither the event defining $I_2(t)$, nor $J_2(t)$ can hold.

Therefore,
\begin{eqnarray*}
\max (I_2(t), J_2(t))& \leq &
\mmp\{\underset{[(1-\epsilon)t]+1\leq i\leq
[(1+\epsilon)t]}{\min}U_i<\underset{1\leq i\leq \lfloor (1-\epsilon)t\rfloor}{\min}\,U_i\}\\
& =&
1-\dfrac{\lfloor(1-\epsilon)t\rfloor}{\lfloor(1+\epsilon)t\rfloor}.
\end{eqnarray*}
By a large deviation result (see, for example, \cite{Bah}), there
exist positive constants $\delta_1$ and $\delta_2$ such that for
all $t>0$
$$I_3(t)\leq \delta_1e^{-\delta_2t}.$$ Select now $t$ such that
$(1-\epsilon)t=n\in\mn$. Then from the calculations above we
obtain
\begin{eqnarray*}
\mmp\{L(n/(1-\epsilon))/a(n)>x\}\leq
\mmp\{L_n/a(n)>x\}+1-n/[(1+\epsilon)n/(1-\epsilon)]+
\delta_1\exp^{-\delta_2n/(1-\epsilon)}.
\end{eqnarray*}
Since $a(t)$ is slowly varying, sending first $n\uparrow \infty$
and then $\epsilon\downarrow 0$ we obtain
$$\underset{n\to\infty}{\lim\inf}\,\mmp\{L_n/a(n)>x\}\geq
\mmp\{Z_{\alpha,\beta}>x\}$$ for all $x>0$ (since the law of
$Z_{\alpha,\beta}$ is continuous). The same argument applied to
\eqref{1} establishes the converse inequality for the upper limit.
The proof of the theorem is herewith complete.

\vskip0.3cm

\section{Number of zero increments of a nonincreasing Markov chain and proof of Proposition
\ref{mix}}\label{mar}

With $M\in\mn_0$ given and any $n\geq M$, $n\in\mn$, let
$(Y_k(n))_{k\in\mn_0}$ be a nonincreasing Markov chain with
$Y_0(n)=n$, state space $\mn$ and transition probabilities
\begin{eqnarray*}
\mmp\{Y_k(n)=j|Y_{k-1}(n)=i\}& = &s_{i,j},\;\;i\geq M+1 \ \ \text{and either} \ M< j\leq i \ \text{or} \ M=j<i,\\
\mmp\{Y_k(n)=j|Y_{k-1}(n)=i\}& = &0,\;\;i< j,\\
\mmp\{Y_k(n)=M|Y_{k-1}(n)=M\}& = &1.
\end{eqnarray*}
Denote by
$$Z_n:=\#\{k\in\mn_0: Y_k(n)-Y_{k+1}(n)=0, Y_k(n)>M\}$$
the number of zero decrements of the Markov chain before the
absorption. Assuming that, for every $M+1\leq i\leq n$,
$s_{i,\,i-1}>0$, the absorption at state $M$ is certain, and $Z_n$
is a.s.\,finite.

Since $L_n$ is the number of zero decrements before the absorption
of a Markov chain with $M=0$ and $$s_{i,j}={i\choose j}\me
W^j(1-W)^{i-j},$$ Proposition \ref{mix} is an immediate
consequence of the following result.
\begin{assertion}\label{mix1}
If $Z_n\dod Z$, $n\to\infty$, where a random variable $Z$ has a
proper and nondegenerate probability law then this law is mixed
Poisson.
\end{assertion}
\begin{proof}
While the chain stays in state $j>M$ the contribution to $Z_n$ is
made by a random variable $R_j$ which is equal to the number of
times the chain stays in $j$, hence the representation
\begin{equation}\label{rep}
Z_n=\sum_{k\geq 0}R_{\widehat{Y}_k(n)}1_{\{\widehat{Y}_k(n)>M\}},
\end{equation}
where $(\widehat{Y}_k(n))_{k\in\mn_0}$ is the corresponding to
$(Y_k(n))$ {\it decreasing} Markov chain with $\widehat{Y}_0(n)=n$
and transition probabilities
$$\widehat{s}_{i,j}={s_{i,j}\over 1-s_{i,\,i}}, \ \ i>j\geq M.$$ It is clear that
$(R_j)_{M+1\leq j\leq n}$ are independent random variables which
are independent of $(\widehat{Y}_k(n))$, and $R_j$ has the
geometric distribution with success probability $1-s_{j,j}$, i.e.,
$$\mmp\{R_j=m\}=(1-s_{j,j})s^m_{j,j}, \ \ m\in\mn_0.$$ Let $(\pi_t)_{t\geq
0}$ be a unit rate Poisson process which is independent of an
exponentially distributed random variable $T$ with mean
$1/\lambda$. Then $\pi_T$ has the geometric distribution with
success probability $\lambda/(\lambda+1)$. Conditioning in
\eqref{rep} on the chain and using the latter observation along
with the independent increments property of Poisson processes lead
to the representation
$$Z_n\od \pi^\ast\bigg(\sum_{k\geq
0}T_{\widehat{Y}_k(n)}1_{\{\widehat{Y}_k(n)>M\}}\bigg),$$ where
$(T_j)_{M+1\leq j\leq n}$ are independent random variables which
are independent of $(\widehat{Y}_k(n))$, and $T_j$ has the
exponential distribution with mean $s_{j,j}/(1-s_{j,j})$, and
$(\pi^\ast(t))_{t\geq 0}$ is a unit rate Poisson process which is
independent of everything else. Since $Z_n$ converges in
distribution, the sequence in the parantheses must converge, too,
and the result follows.
\end{proof}
As a generalization of the Fact 2 stated on p.~\pageref{fa} we
will prove the following
\begin{assertion}\label{geo}
If $s_{j,M}=s_{j,j}$, for every $j\geq M+1$, then, for every
$n\geq M$, $Z_n$ has the geometric law starting at zero with
success probability $1/2$.
\end{assertion}
\begin{proof}
By the assumption, $s_{M+1, M}=s_{M+1, M+1}=1/2$, and the result
for $Z_{m+1}$ follows from \eqref{rep}.

Assume the statement holds for $M+1\leq k<n$, i.e.,
$\mmp\{Z_k=j\}=2^{-j-1}$, $j\in\mn_0$. Conditioning on $Y_1(n)$
and noting that $Z_M=0$ we conclude that
$$\mmp\{Z_n=0\}=s_{n,M}+\sum_{k=M+1}^{n-1}\mmp\{Z_k=0\}s_{n,k}=s_{n,M}+1/2(1-s_{n,M}-s_{n,n})=1/2.$$
Similarly, for $j\in\mn$,
\begin{eqnarray*}
\mmp\{Z_n=j\}&=&\mmp\{Z_n=j-1\}s_{n,n}+\sum_{k=M+1}^{n-1}\mmp\{Z_k=j\}s_{n,k}\\&=&
\mmp\{Z_n=j-1\}s_{n,n}+1/2(1-s_{n,M}-s_{n,n}).
\end{eqnarray*}
Using this for $j=1$ and recalling that $\mmp\{Z_n=0\}=1/2$ gives
$\mmp\{Z_n=1\}=1/4$. Repeating this argument for consecutive
values of $j$'s completes the proof.
\end{proof}

Now we want to apply the results of this section to another
nonincreasing Markov chain. Let $(\tau_k)_{k\in\mn}$ be
independent copies of a random variable $\tau$ with distribution
$$p_j:=\mmp\{\tau=j\}, \ j\in\mn, \ \ p_1>0.$$ The {\it random
walk with barrier} $n\in\mn$ (see \cite{IksMoe2} for more details)
is a sequence $(W_k(n))_{k\in\mn_0}$ defined as follows:
$$W_0(n):=0 \ \ \text{and} \ \
W_k(n):=W_{k-1}(n)+\tau_k 1_{\{W_{k-1}(n)+\tau_k<n\}}, \ \
k\in\mn.$$ Plainly, $(n-W_k(n))_{k \in \mn_0}$ is a nonincresing
Markov chain with $$s_{i,j}=p_{i-j}, \ i>j, \ \ \text{and} \ \
s_{i,i}=\sum_{k\geq i}p_k,$$ which starts at $n$ and eventually
get absorbed in the state $1$. Let $Z_n^\ast$ be the number of
zero decrements before the absorption of this chain. It was shown
in \cite[Theorem 1.1]{IksNeg} that, provided $\me \tau<\infty$,
$Z_n^\ast$ converges in distribution. By the virtue of Proposition
\ref{mix1}, this can be complemented by the statement that the
limiting law is mixed Poisson. Finally, using Proposition
\ref{geo} we conclude that, for every $n\geq 2$,
$$\mmp\{Z_n^\ast=m\}=2^{-m-1}, \ \ m\in\mn_0,$$ provided
$p_j=2^{-j}$, $j\in\mn$.


\begin{thebibliography}{99}
\footnotesize


\bibitem{Ath} {\sc Anderson, K.~K. and Athreya, K.~B.} (1987). A renewal theorem in the infinite mean case. {\em Ann. Probab.} {\bf 15},
388--393.

\bibitem{Bah} {\sc Bahadur, R.R.} (1971). Some limit theorems in
statistics. CBMS Regional conference series in applied
mathematics. {\bf 4}. Philadelphia: SIAM.

\bibitem{BerYor} {\sc Bertoin, J. and Yor, M.} (2005). Exponential functionals
of L\'{e}vy processes. {\em Probability Surveys}. {\bf 2},
191--212.

\bibitem{Bing73} {\sc Bingham, N.~H.} (1973). Maxima of sums of random variables and suprema of stable
processes. {\em Z. Wahrsch. verw. Gebiete.} {\bf 26}, 273-- 296.






\bibitem{BGT}{\sc Bingham N.~H., Goldie C.~M., and Teugels, J.~L.} (1989).
Regular variation. Cambridge: Cambridge University Press.

\bibitem{Durr} {\sc Durrett, R. and Liggett, T.~M.} (1983).
Fixed points of the smoothing transformation. {\em  Z. Wahrsch.
Verw. Gebiete.} {\bf 64}, 275--301.

\bibitem{Feller} {\sc Feller, W.} (1971). An introduction to
probability theory and its applications, Vol.~2, 2nd edition, John
Wiley \& Sons, New York etc.
%
%


\bibitem{GneIksMar} {\sc Gnedin, A., Iksanov, A. and Marynych, A.} (2010).
Limit theorems for the number of occupied boxes in the Bernoulli
sieve. {\em Theory of Stochastic Processes}. {\bf 16(32)}, 44--57.

\bibitem{GIM2} {\sc Gnedin, A., Iksanov, A. and Marynych, A.} (2010).
The Bernoulli sieve: an overview. {\em Discr. Math. Theoret.
Comput. Sci.} Proceedings Series, {\bf AM}, 329--342.

\bibitem{GINR}{\sc Gnedin, A., Iksanov, A., Negadajlov, P., and Roesler, U.} (2009).
The Bernoulli sieve revisited. {\it Ann.\,Appl.\,Prob.} {\bf 19},
1634--1655.

\bibitem{Yor} {\sc Gradinaru, M., Roynette, B., Vallois, P. and Yor, M.} (1999).
Abel transform and integrals of Bessel local times. {\em Ann.
Inst. Henri Poincar\'{e}}. {\bf 35}, 531--572.

\bibitem{IksNeg} {\sc Iksanov, A. and Negadajlov, P.} (2008). On the
number of zero increments of random walks with a barrier. {\em
Discrete Mathematics and Theoretical Computer Science},
Proceedings Series. {\bf AG}, 247--254.


%
%



\bibitem{IksMoe2}{\sc Iksanov, A. and M\"ohle, M.} (2008). On the
number of jumps of random walks with a barrier. {\em Adv. Appl.
Probab.} {\bf 40}, 206--228.

\bibitem{Kaplan} {\sc Kaplan, N.} (1975). Limit theorems for a
$GI/G/\infty$ queue. {\em Ann. Probab.} {\bf 3}, 780--789.


\bibitem{Konst} {\sc Konstantopoulos, T. and Lin. S.} (1998). Macroscopic models for long-range dependent network
traffic. {\em Queueing Systems.} {\bf 28}, 215--243.

\bibitem{Mag} {\sc Magdziarz, M.} (2009). Stochastic representation of subdiffusion processes with
time-dependent drift. {\em Stoch. Proc. Appl.} {\bf 119},
3238--3252.

\bibitem{Res} {\sc Mikosch, T. and Resnick, S.} (2006). Activity
rates with very heavy tails. {\em Stoch. Proc. Appl.} {\bf 116},
131--155.





\end{thebibliography}
\end{document}